\documentclass[12pt]{mymathart}
\usepackage{mymathmacros}

 \newtheoremstyle{numberedstyle}% name
   {9pt}%      Space above, empty = `usual value'
   {9pt}%      Space below
   {\normalfont}% Body font
   {}%         Indent amount (empty = no indent, \parindent = para indent)
   {\bfseries}% Thm head font
   {.}%        Punctuation after thm head
   {\newline}% Space after thm head: \newline = linebreak
   {}%         Thm head spec

\newtheorem{thm}{Theorem}[section]%
\newtheorem{lem}[thm]{Lemma}%
\newtheorem{cor}[thm]{Corollary}%
\newtheorem{obs}[thm]{Observation}%
\newtheorem{lemdef}[thm]{Lemma and Definition}

\theoremstyle{numberedstyle}
\newtheorem{defn}[thm]{Definition}%

\newcommand{\B}{\mathcal{B}}

\newcommand{\F}{\mathcal{F}}

\usepackage{subfigure,graphicx}

\title[Hyperbolic dimension and radial Julia sets]{Hyperbolic dimension and radial Julia sets \\
 of transcendental functions}
\author{Lasse Rempe}
\thanks{The author was supported by EPSRC grant EP/E017886/1.}

\renewcommand{\A}{\mathcal{A}}
\renewcommand{\P}{\mathcal{P}}
\newcommand{\hyp}{\operatorname{hyp}}
\newcommand{\Comp}{\operatorname{Comp}}

\begin{document} 
 \begin{abstract}
  We survey the definition of the radial Julia set $J_r(f)$ 
   of a meromorphic function
   (in fact, more generally, any \emph{Ahlfors islands map}),
   and give a simple
   proof that the Hausdorff dimension of $J_r(f)$ and the
   hyperbolic dimension $\dim_{\hyp}(f)$ always coincide.
 \end{abstract}
 \maketitle

 \section{Introduction}
  The study of the measurable dynamics of rational functions is 
   well-developed. In particular, it is known \cite{przytyckiconical,denkerurbanskisullivanconformal}
   that a large number
   of dynamically natural quantities coincide 
    (compare also the
   survey article \cite{urbanskisurvey}):
   \begin{enumerate}
    \item The Hausdorff dimension of the \emph{radial} 
     (also called \emph{conical}) 
     \emph{Julia set}
     $J_r(f)$, introduced by
     Urbanski \cite{urbanskiconical} and McMullen \cite{mcmullenradial}
     as the set of points where it is possible to go from small
     to big scales via univalent iterates. 
     (See Definition \ref{defn:radialjulia}.)
    \item The \emph{hyperbolic dimension}, introduced by Shishikura
      \cite{mitsudim} as the supremum of the Hausdorff dimensions
      of hyperbolic subsets of $J(f)$. (See Definition \ref{defn:hyperbolicdimension}.)
    \item The \emph{dynamical dimension}, defined by Denker and Urbanski
      \cite{denkerurbanskisullivanconformal}
      as the supremum over the Hausdorff dimensions of all ergodic
      invariant probability measures of positive entropy.
    \item The minimal exponent for which a \emph{conformal measure}
      in the sense of Sullivan exists.
   \end{enumerate}
  In well-controlled cases, these numbers also agree with the
   Hausdorff dimension of $J(f)$, but it is unknown whether this
   holds in general
   (for transcendental functions, this is often false 
    \cite{stallardhyperbolicmero,urbanskizdunik1}, and
    results of Avila and Lyubich \cite{avilalyubichfeigenbaum2} also suggest
    the existence of rational functions for which this equality fails). 
 
  In recent years, there have been a large number of cases where the
   above program was also carried through for
   transcendental entire and meromorphic functions. 
   (We refer the reader to the survey \cite{kotusurbanskisurvey}.) 
   This suggests the question which parts of the general rational theory
   carries through for \emph{arbitrary} 
   transcendental meromorphic functions. 
   The purpose of this note is to observe the following 
   (where $\dim_H$ denotes Hausdorff dimension). 

  \begin{thm}[Equality of hyperbolic and radial dimensions]
     \label{thm:main}
    Let $f:\C\to\Ch$ be any nonconstant and nonlinear meromorphic
     function. Then 
        \[ \dim_{\hyp}(f) = \dim_H( J_r(f)). \] 
  \end{thm}
  \begin{remark}
   In fact, we prove this theorem even more generally for the class of
    \emph{Ahlfors islands maps} studied by Epstein 
    \cite{finitetype} and Oudkerk \cite{oudkerkahlfors} (see 
    Definition \ref{defn:ahlforsislands} below). 
  \end{remark}

  Previously, this equality was usually established using 
   conformal measures of the correct exponent 
   for the function $f$. However,
   the construction of such measures seems more difficult
   in the transcendental case, and, in currently
   known cases, uses some
   additional dynamical and function-theoretic assumptions. 

  Our
   main observation is that it is possible to avoid this step, and
   construct a suitable hyperbolic set directly from the definition of 
   the radial Julia set. Even for rational functions, this argument 
   (though not the theorem) seems to be new. We believe the construction
   is interesting due to its simplicity and wide generality.

  Estimating the hyperbolic dimension of $f$ (and hence the dimension of
   $J_r(f)$) is more difficult. Stallard 
   has shown that the Julia set of
   a transcendental meromorphic function can have any Hausdorff dimension
   in the interval $(0,2]$, and that for a transcendental entire function
   this dimension can be any number in the interval $(1,2]$. 
   While she proved \cite{stallardmeromorphichausdorff} that
   the hyperbolic
   dimension is always strictly greater than $0$ (the proof carries
   over directly to Ahlfors islands maps, and we sketch it in Corollary
   \ref{cor:dimensiongreaterthanzero} below), 
   it is not known whether the Hausdorff dimension of the Julia set of
   an entire transcendental function can equal the lower bound $1$.

  For functions with a \emph{logarithmic singularity} over infinity, though, 
   this bound cannot be achieved 
   (this was proved originally in
   \cite{stallardentirehausdorff} 
   for entire functions in the \emph{Eremenko-Lyubich class}
   $\B$, and subsequently strenghtened in 
   \cite{ripponstallardhausdorffmero,walterphilgwyneth}). Recently, 
   Bara{\'n}ski, Karpi{\'n}ska and Zdunik \cite{baranskikarpinskazdunik} 
   extended this result
   to show also that the hyperbolic dimension of such a function
   is strictly larger than one. We discuss a minor generalization of this
   fact in Theorem \ref{thm:logtracts}. 

 \subsection*{Structure of the article}
   In Section \ref{sec:defns}, we review a number
    of definitions which do not explicitly appear in the literature in this
    generality. 
   The beautiful
    concepts of finite-type and 
    Ahlfors islands maps, both due to
    Adam Epstein, have appeared only in the 
    unpublished references \cite{adamthesis,finitetype,oudkerkahlfors}, 
    which are not widely available,  
    so we take the 
    opportunity to introduce them here.  
   We also define radial Julia sets,
    whose well-known definition in the rational case carries over 
    verbatim but does not seem to have been formally stated in
    the transcendental setting.
    We also review the definition of hyperbolic dimension, and
    for the reader's convenience include a short summary on an elementary
    type of conformal
    iterated function system which will be used to construct hyperbolic
    sets. Finally, we prove Theorem
    \ref{thm:main} in Section \ref{sec:proof}, and discuss
    hyperbolic dimension in logarithmic tracts in an appendix.

 \subsection*{Notation}
  Throughout the rest of this article, $X$ is a compact Riemann surface
   and $W\subset X$ is a nonempty open subset.
   A reader who is interested only in the classical cases may wish to
    assume that
    $W=\C$ and $X=\Ch$ (for transcendental meromorphic functions)
    or $W=X=\Ch$ (for rational functions), and skip the
    subsection on Ahlfors islands maps in Section \ref{sec:defns}.

   In what follows, all distances, balls, derivatives etc.\ are 
    considered with respect to the 
     natural
     conformal metric on the Riemann surface $X$ 
    (i.e., the unique metric of constant curvature 
     $-1$, $0$ or $1$). Since $X$ is compact, for sufficiently small
     $\delta>0$, any disk
        \[ \D_{\delta}(z_0) = \{ z\in\ X: \dist(z_0,z)<\delta \} \]
     is simply connected with Jordan curve boundary. 
     Recall that, by Rad\`o's theorem, $X$ is second countable, so there
     is a countable base for the topology of $X$ consisting of such
     simply-connected balls. 

   If $A$ is a topological space and $a\in A$, then we denote by
    $\Comp_a(A)$ the connected component of $A$ containing $a$. 

  Finally, 
   we remark that ``meromorphic'' for us always includes ``entire'' as a 
   special case. 
 
 \subsection*{Acknowledgments}
   I would like to thank Boguslawa Karpinska, Janina Kotus, Feliks
   Przytycki, Gwyneth Stallard and Anna Zdunik for interesting discussions. 

\section{Definitions} \label{sec:defns}

 \subsection*{Ahlfors islands maps}
 
  The class of \emph{Ahlfors islands maps} was suggested by
   Adam Epstein \cite{finitetype} as a 
   generalization of both general transcendental meromorphic functions
   and his ``finite-type maps'' (see below), and studied in
   greater detail by him and Richard Oudkerk \cite{oudkerkahlfors}. 
   This class 
   provides a general 
   setting for the study of transcendental dynamics. We remark that
   a similar condition for a restricted class of maps (those defined
   outside a small subset of the Riemann sphere) was considered
   by Baker, Dominguez and Herring in \cite{bakerdominguezherring}.

 \begin{defn}[Ahlfors Islands maps] \label{defn:ahlforsislands}
 Let $W\subset X$ be open and nonempty. 
  We say that a nonconstant holomorphic function
  $f:W\to X$ is an \emph{Ahlfors islands map} if there is a finite number
  $k$ satisfying the following condition (known as the ``islands property'').
  \begin{center}
   \emph{Let $V_1,\dots,V_k\subset X$ be Jordan domains with pairwise disjoint
    closures, and let $U\subset X$ be open and connected with
    $U\cap \partial W\neq\emptyset$. Then for every component $U_0$ of 
    $U\cap W$,
    there is $i\in \{1,\dots,k\}$ such that $f$ has an island over $V_i$ in
    $U_0$. (That is, there is a
    domain $G\subset U_0$ such that $f:G\to V_i$ is a conformal isomorphism.)}
  \end{center}

  An Ahlfors islands map is \emph{elementary} if $W=X$ and
   $f:W\to X$ is a conformal isomorphism. 

  The set of all non-elementary Ahlfors islands maps from $W$ to $X$ is
   denoted $\A(W,X)$. We also set
      \begin{align*}
        \A(X) &:= \bigcup\left\{\A(W,X):{W\subset X \text{ open}}\right\}
           \quad\text{and}\\
        \A &:=  \bigcup \left\{ \A(X): {X \text{ is a compact Riemann Surface}}\right\}. 
      \end{align*}  
  \end{defn}

  An important feature of the class of Ahlfors islands maps is that it
   is closed under composition. We note that
   nonconstant rational functions are Ahlfors islands maps by
   definition. The name of this class is derived from the classical
   five-islands theorem \cite{walterahlfors}:

  \begin{thm}[Ahlfors Five Islands Theorem] \label{thm:ahlfors}
   Every transcendental meromorphic function $f:\C\to\Ch$ is an element
    of $\A(\C,\Ch)$ (with $k=5$). 
  \end{thm}
  
 There is also a 
   \emph{normal families version} 
   \cite[Theorem A.1]{walterahlfors} of Theorem \ref{thm:ahlfors},
    which we use a number of times below: 
   if $\F$ is a family of 
   analytic functions which have no islands over given 
   domains $V_1,\dots,V_n$ as above, then $\F$ is normal. 
 
  Another interesting subclass of $\A(W,X)$ is that of \emph{finite-type maps},
   first introduced in \cite[Chapter 3]{adamthesis}, which provides 
   a natural generalization of the dynamics
   of rational functions and meromorphic functions with finitely many
   singular values. A nonconstant analytic function
   $f:W\to X$ is a \emph{finite-type map} if it has
   only finitely many singular values in $X$, and furthermore there are no
   isolated removable singularites in $\partial W$. 

  \begin{thm}[Epstein]
   If $f:W\to X$ is a finite-type map, then $f$ is an Ahlfors islands
    map.
  \end{thm}

 There are many examples of Ahlfors islands maps which do not fall into 
  either of the previous two classes. For example. a ``Valiron function''
  is an analytic map $f:\D\to\C$ which is unbounded in $\D$, but bounded
  on a ``spiral'', i.e., on 
  a curve $\gamma(t) = r(t)\cdot e^{2\pi i \theta(t)}$ with
   $r(t)\nearrow 1$ and $|\theta(t)|\to\infty$. Every such function is
  an Ahlfors map (see \cite[Remark in Section 4]{ripponbarthvaliron}) ---
  indeed, these functions are not normal near any point of $\partial\D$,
  and the result follows by the normal families version of the Ahlfors
  five islands theorem. For the same reason, other functions for which
  normality fails in this way --- e.g.\ the \emph{weak Valiron functions}
  considered in \cite{ripponbarthvaliron} --- also belong to the class
  of Ahlfors islands maps. 

 \smallskip

  The Fatou set $F(f)\subset X$
   of an Ahlfors islands map $f:W\to X$ is now defined
   as usual, except that, 
   in addition to those points which have a neighborhood
   where all iterates are defined and 
   form a normal family, 
   it will also include all 
   points with $f^j(z)\in X\setminus \cl{W}$
   for some $j$. Thanks to the Ahlfors islands property,
   Baker's proof of the density of repelling periodic points
   goes through as usual, and $J(f)$ has all the standard properties.

  \begin{thm}[Baker, Epstein] \label{thm:juliaproperties}
   Let $f\in \A(W,X)$. Then $J(f)$ is a nonempty
    perfect subset of $X$, and repelling periodic points of $f$ are dense
    in $J(f)$. Furthermore, if $J(f)\neq X$, then $J(f)$ has empty interior. 
  \end{thm}

 Note that $J(f)\neq\emptyset$ is trivial by definition as soon as 
   $\partial W\neq\emptyset$, while $\partial W = \emptyset$ means that we
   are in the well-known case of a rational self-map of the Riemann sphere
   or of a self-map of the torus. 
  The key fact in the proof of the density of repelling periodic
   points is the 
   following lemma, which we will use below and whose proof we sketch
   for completeness. The
   remaining claims in Theorem \ref{thm:juliaproperties}
   follow in the usual manner. 

 \begin{lem}[Islands near the Julia set] \label{lem:modifiedislands}
  Let $f\in \A(W,X)$. Then there is a number $k$ with the following property:
   if $V_1,\dots,V_k\subset X$ are Jordan domains with disjoint closures,
   then there is some $V_j$ such that every open set $U$ with 
   $U\cap J(f)\neq\emptyset$ contains an island of $f^n$ over $V_j$, for some
   $n$. 
 \end{lem}
 \begin{proof}
  Let $k\geq 5$ be as in the islands property. First suppose that
   $\partial W\neq \emptyset$; then by the islands property, 
   $f$ has infinitely many islands over some $V_j$; in particular, we 
   can pick $k$ such islands $G_1,\dots,G_5$.
  
  If $\partial W = \emptyset$, then $f$ must be either a rational function
   or a self-map of the torus; in either case, it is easy to see that
   there is some $V_j$ and $k$ domains $G_1,\dots,G_k$ which are 
   islands of some \emph{iterates} of $f$. (For example, this follows
   by using the fact that $J(f)\neq\emptyset$ and the 
   normal families version of the
   Ahlfors five-islands theorem.)

  Now let $U$ be as in the statement of the lemma.
   If all iterates of $f$ are defined on $U$, then the normal families
   version of the Ahlfors five-islands theorem implies the existence
   of an island of some iterate of $f$ over one of the domains $G_1,\dots,G_k$,
   which is hence also an iterated island over the original domain $V_j$. 
   
  Otherwise, there is some $z_0\in U$ with $f^n(z_0)\in \partial W$. By 
   shrinking
   $U$, if necessary, we may assume that $f^n|_U$ is a proper map with
   no critical points except possibly $z_0$. By the islands property, 
   $f^n(U)$ contains an island of $f$ over one of the $G_1,\dots, G_k$, 
   which then pulls back to an island of $f^{n+2}$ over $V_j$, as desired.
 \end{proof}

 \begin{remark}
  Epstein's original definition of finite-type and Ahlfors islands maps
   allows $X$ to be the disjoint union of several compact Riemann surfaces,
   in which case the concept of ``elementary'' cases becomes slightly
   more involved, but everything else goes through as in the connected case.

  For finite-type maps, Epstein proves the absence of
   Baker and wandering domains as well as the finiteness of nonrepelling
   cycles.  
   We refer the reader to the thesis \cite{adamthesis} and the
   manuscript \cite{finitetype} 
   for further information.
 
  We remark that 
   Oudkerk \cite{oudkerkahlfors}
   suggested a slightly modified definition of Ahlfors islands maps
   which is equivalent to the above, but ensures that, for each $k$,
    the class of
   ``Ahlfors $k$-islands maps'' is closed under composition.
   (With our definition, this is true only for $k\geq 5$.) 
 \end{remark}

 \subsection*{Radial Julia sets} 
  \begin{defn}[Radial Julia sets] \label{defn:radialjulia}
   Let
    $f\in \A(W,X)$.
    The
    \emph{radial Julia set}
    $J_r(f)$ is the set of all points $z\in J(f)$ 
     with the following property: all forward images of $z$ are defined, and
     there is some $\delta>0$ such that, for infinitely many $n\in\N$, the 
     disk $\D_{\delta}(f^n(z))$ can be pulled back univalently along the
     orbit of $z$. (I.e., 
       \[ f: \Comp_{z}\bigl( \D_{\delta}(f^n(z))\bigr) \to
                             \D_{\delta}(f^n(z)) \]
     is a conformal isomorphism.)
  \end{defn}

  Recall that all disks and distances
   are understood with respect to the metric on the compact
   surface $X$ 
   (i.e., the spherical metric in the case of rational, entire or
    meromorphic functions, where $X=\Ch$). 
   We remark that, if $z\in J(f)$ and 
   $\limsup \dist(f^n(z),\P(f))>0$, where $\P(f)$ is the postsingular set
   of $f$, then $z\in J_r(f)$.

  \begin{lemdef}[Disks of Univalence] \label{lemdef:disksofunivalence}
   Let $z_0\in J_r(f)$. Then there exists
    a disk 
    $D = \D_{\delta}(w_0)\subset X$ with the following properties.
  \begin{enumerate}
   \item 
     Both $D$ and $\D_{2\delta}(w_0)$ are simply connected.
   \item  
    for infinitely many $n$, $f^n(z_0)\in D$ and 
    $\D_{2\delta}(w_0)$ pulls back univalently to $z_0$
     under $f^n$.
  \end{enumerate}
  In particular, if $n$ is as above and 
   $U$ is the component of $f^{-n}(D)$ containing $z_0$, then 
    \[
      \theta := \sup_{z\in U} \|Df^n(z)\|<\infty, \quad\text{and}\quad
           U \subset \D_{C/ \theta}(z_0). \]
  (Here $C$ depends only on $X$.) 
  Furthermore, $\theta\to \infty$ as $n\to\infty$.

  We will call such a disk $D$ a \emph{disk of univalence} for $z_0$. 
 \end{lemdef}
% \begin{remark}
%  This is a slightly different definition from what is called
%    a disk of univalence in \cite{linefields}.
% \end{remark}
 \begin{proof}
   Let $\delta'$ be the constant from the definition of $J_r(f)$. 
    Let $w_0$ be a limit point of a sequence $z_n = f^n(z_0)$ of iterates 
    for which $\D_{\delta'}(z_n)$ can be pulled back univalently.
   Then we set $D := \D_{\delta}(w_0)$, where $\delta < \delta/2$ is
    sufficiently small so that $D$ and $\D_{2\delta}(w_0)$ 
    are simply connected. 

   The final claim follows from the Koebe distortion theorem.
 \end{proof} 

 \subsection*{Hyperbolic dimension}
  Hyperbolic dimension was introduced by Shishikura in his proof
    \cite{mitsudim} that the boundary of the Mandelbrot set has Hausdorff 
     dimension two. 
   \begin{defn}[Hyperbolic dimension] \label{defn:hyperbolicdimension}
    Let $f\in \A(W,X)$. A \emph{hyperbolic set} is a compact, forward-invariant
     set $K\subset W$ such that $\|Df^n|_K\| > \lambda > 1$
      for some $n$ and $\lambda$. 
   
    The \emph{hyperbolic dimension} of $f$ is the supremum of
     $\dim_H(K)$ over all hyperbolic sets. 
   \end{defn}

 The following fact is well-known (compare
    \cite[Lemma 3.4]{linefields}).
  \begin{lem}[Hyperbolic sets and the radial Julia set] 
       \label{lem:hyperbolicsets}
    Any hyperbolic set is contained in $J_r(f)$. Conversely, any
     compact and forward invariant subset of $J_r(f)$ is hyperbolic.

    In particular,
     the hyperbolic dimension $\dim_{\hyp}(f)$ is no larger than
     the Hausdorff dimension of $J_r(f)$.  \qedd
  \end{lem}

 \subsection*{Conformal iterated function systems}
  Hyperbolic sets are usually obtained by constructing a conformal
   iterated function system consisting of iterates of the function $f$. 
   We review the basic case of the theory for the reader's convenience. 
  \begin{defn}
    Let $V\subset X$ be simply connected. Also let 
     $U_1,\dots, U_m\subsetneq V$ be pairwise
     disjoint simply connected domains with
     $\cl{U_j}\subset V$, and let
     $g_j:V\to U_j$ be conformal isomorphisms. 
     Then we say that $(g_j)$ is an
   \emph{elementary
     conformal iterated function
     system}. 

   The \emph{limit set} of the system $(g_j)$ is the set
      \[ L := \bigcap_{n\geq 0} 
                \bigcup_{(j_1,\dots,j_n)} 
                  g_{j_n}\bigl(\dots(g_{j_1}(V)\bigr). \]
  \end{defn}

 (By the Schwarz lemma, the maps $g_j$ are contractions, and hence
  $L$ is a compact, totally disconnected set.)

  In the well-known case where all $g_j$ are
   affine similarities, the Hausdorff dimension of the
   limit set is the unique value $t$ with
     \begin{equation}
     \sum (\lambda_j)^t = 1,  \label{eqn:similaritydimension}
     \end{equation}
   where $|g_j'| \equiv \lambda_j <1$ \cite{falconerfractalgeometry}. 
   We will use the
   following lemma in the general case. 

  \begin{lem}[Hausdorff dimension of limit sets] \label{lem:limitsets}
    Suppose that $L$ is the limit set of a conformal iterated function system
     $(g_j)$ as above.  
    Set $\lambda_j := \inf_{z\in V} \|Dg_j(z)\|$, and suppose that
     \[ \sum {\lambda_j}^t \geq 1. \]
 
    Then $\dim_H(L)\geq t$. \qedd
  \end{lem}

  In fact, the Hausdorff dimension of the limit set can be computed exactly
   by generalizing the formula (\ref{eqn:similaritydimension}) to the
   more general setting as follows. We set 
     \begin{equation}
      \label{eqn:pressure}
         P(t) := 
         \lim_{n\to \infty}\frac{1}{n}  \log\left(
       \sum_{(j_1,\dots,j_n)} \|D(g_{j_1}\circ\dots\circ g_{j_n})(z_0)\|^t\right), 
     \end{equation}
     where $z_0\in V$ is some base point. 
     (The existence of the limit
     and its independence from $z_0$ follow from the Koebe distortion
     theorem.)

   Then the Hausdorff dimension $t$ is the unique zero of the strictly
    decreasing function
    $P(t)$ (usually called the ``pressure function''). 
    (For a proof, in the more general setting of
     \emph{infinite} conformal iterated function systems in any
     dimension, see \cite{urbanskimauldin}.) 

   We obtain 
    the above lemma
    as an obvious corollary, since the terms in the limit in
    (\ref{eqn:pressure}) are bounded from below by    
  \[ \frac{1}{n} \log\left( \left(\sum \lambda_j^t\right)^n\right)
                       = \log\left(\sum \lambda_j^t\right). \]

 \begin{cor}[Hyperbolic dimension of Ahlfors islands maps]
     \label{cor:dimensiongreaterthanzero}
  Let $f\in \A(W,X)$ be a non-elementary Ahlfors islands map. Then
   $\dim_{\hyp}(f) > 0$. 
 \end{cor}
 \begin{proof}
   Let $k$ be as in Lemma \ref{lem:modifiedislands}, and pick 
    $k$ Jordan domains $U_1,\dots,U_k$, each of which intersects
    the Julia set. Then for at least one of these domains,
    say $U = U_j$, there are islands of some iterate of $f$ over $U$
    near every point of the Julia set. Using the fact that $J(f)$ is perfect,
    we can thus find two (in fact infinitely many) domains
    $V_1,V_2\Subset U$, with $\cl{V_1}\cap \cl{V_2}=\emptyset$, such 
    that some iterate of $f$ maps $V_j$ conformally to $U$. 
 
   The corresponding inverse branches $g_j:U\to V_j$ then form an
    elementary conformal iterated function system, and the
    limit set $L$ of this system has Hausdorff dimension strictly
    greater than $0$. Clearly $L$, together with its finitely many forward
    images is a hyperbolic set for $f$, and the claim follows. 
 \end{proof}

 \subsection*{Hausdorff dimension and backward invariance}
   We include a proof of the following simple lemma. 
  \begin{obs}[Hausdorff dimension of backward invariant sets]
      \label{obs:dimension}
    Let $f\in \A(W,X)$ be a non-elementary Ahlfors islands map, and let
     $B\subset X$ be backward invariant under $f$; i.e. $f^{-1}(B)\subset B$. 

     If $U\subset X$ is an open set intersecting $J(f)$, then 
      $\dim_H(B\cap U) = \dim_H(B)$. 
  \end{obs} 
  \begin{proof}
   It follows easily from the definition of Hausdorff dimension that
    there is a sequence $D_1,D_2,\dots$ of Jordan disks with 
    pairwise disjoint closures such that 
       $\dim_H(B \cap D_j) \to \dim_H(B)$.

   If $U$ contains some point where $f^j$ is undefined for some $j\geq 0$,
    then it follows from the Ahlfors islands property that 
    $U$ contains an island over $D_j$ for infinitely many $j$. Otherwise,
    this follows from the usual five-islands theorem for (non-)normal
    families. 

   So in either case (since conformal maps preserve Hausdorff dimension), 
    we have
     \[ \dim_H(B\cap U) \geq \dim_H(B\cap D_j) \]
    for infinitely many $j$, which proves the claim.
  \end{proof} 

 \section{Hyperbolic dimension and radial Julia sets}  \label{sec:proof}

  \begin{thm}
   Let $f\in \A(W,X)$ be any non-elementary Ahlfors islands map. 
   Then 
    \[ \dim_{hyp}(f) = \dim_H(J_r(f)). \]
  \end{thm}
  \begin{proof}
   The inequality ``$\leq$'' follows from Lemma \ref{lem:hyperbolicsets}. 
   So set $d:=\dim_H(J_r(f))$ and let $\eps>0$. We shall show that
   $\dim_{\hyp}(f)\geq d' := d-\eps$ by constructing a hyperbolic
   set of dimension at least $d'$.

   Pick a countable basis of the topology of $X$ consisting of
    simply connected disks $D_i$. 
    Let $J_r(D_i)$ be the set of points
    in $J_r(f)$ for which $D_i$ is a disk of univalence
    as defined in Lemma \ref{lemdef:disksofunivalence}. Note that, if
    $z\in J_r(D_i)$ and $w\in W$ with $f(w)=z$, then either $w\in J_r(D_i)$ or
    $w$ is a critical point. So $J_r(D_i)$ is the union of a backward invariant
    set and a countable set, and hence all nonempty open subsets
    of $J_r(D_i)$ have the same Hausdorff dimension by Observation
    \ref{obs:dimension}. 

  Since $J_r(f)$ is the countable union of the $J_r(D_i)$, we can find 
   an index $i$ such that $J_r(D_i)$ has Hausdorff dimension greater than 
   $d'$. Set $D := D_i$ and let $A := J_r(D)\cap D$; then
   $\dim_H(A) > d'$ by the above. Hence, we can pick
   $\delta>0$ such that any covering of $A$ by sets $U_j$ of
   diameter at most $\delta$ has
     \[ \sum_j \diam(U_j)^{d'} > (10\cdot C)^{d'}, \]
   where $C$ is the constant from Lemma \ref{lemdef:disksofunivalence}. 

   Now, for every point $a\in A$, pick a univalent pullback $U_a$
    of $V$ around $a$ as in Lemma \ref{lemdef:disksofunivalence} such that 
     \[ \theta_a :=  \|Df^n|_{U_a}\| > \frac{10C}{\delta} \] 
    and $\cl{U_a}\subset V$. Then $A$ is covered by the open balls
       \[ \D_{C/ \theta_a}(a), \]
    and by a standard covering theorem 
     (see \cite[Theorem 2.1]{mattilageometryofmeasures}
      or \cite[Theorem 2.84 and corollaries]{federermeasuretheory}), 
     we can pick a subsequence
     $a_j$ such that the balls $\D_{C/\theta_{j}}(a_j)$ are disjoint,
     (where $\theta_j = \theta_{a_j}$) 
     while 
       \[ A \subset \bigcup \D_{5C/\theta_{j}}(a_j). \]

   By the condition on the size of
     $\theta_a$, these balls have diameter
    $10C/\theta_j < \delta$. So by choice of $\delta$, we  have
        \[ \sum_{j=1}^{\infty} 
      \left(\frac{10C}{\theta_j}\right)^{d'} > (10C)^{d'}. \]
   If we pick $N$ sufficiently large, set $\lambda_j := 1/\theta_j$ 
    and consider the elementary conformal iterated function system given by the
    branches $g_j$ of $f^{-n_j}$ which take $D$ to $U_j$, we have
       \[ \sum_{j=1}^N \lambda_j^{d'} > 1, \]
    and hence the limit set $L$ of this system satisfies
      \[ \dim_H(L) \geq d' \]
    by Lemma \ref{lem:limitsets}. 
    We have $L\subset A\subset J_r(f)$ by construction, and $L$ together
     with finitely many forward images forms a compact, forward
     invariant set, which is hence a hyperbolic set for $f$. 
  \end{proof}

\appendix

 \section*{Appendix: Hyperbolic dimension in logarithmic tracts} 
    \label{sec:logarithmictracts}
 \stepcounter{section}

  If $U\subset\C$ is an unbounded Jordan domain, and
   $f:U\to\{|z|>R\}$ (where $R>0$) is a universal covering, then
   Bara\'nski, Karpi\'nska and Zdunik \cite{baranskikarpinskazdunik}
   show that there is a hyperbolic set $K\subset U$
   for $f$ with $\dim_H(K)>1$.
   We remark that 
   this carries over to a slightly more general setting where the asymptotic
   curves only \emph{accumulate} on the asymptotic value:

 \begin{thm}[Hyperbolic dimension and logarithmic tracts]
   \label{thm:logtracts}
   Let $X$ be a compact Riemann surface, let
    $D\subset X$ be a Jordan domain, and let $a\in D$. 

   Suppose that $G\subset X\setminus \{a\}$ 
    is a simply connected domain, and that
      \[ f:G \to D\setminus \{a\} \]
   is a universal covering map. Suppose furthermore that there is
   a sequence $(z_n)$ in $G$ with $z_n\to a$ and 
   $f(z_n)\to a$. 

   Then $G$ contains a hyperbolic set for $f$ of dimension greater than one.
     \qedd
 \end{thm}

 In the setting of \cite{baranskikarpinskazdunik}, one
  has the stronger property that $z_n\to a$ whenever $f(z_n)\to a$ in $G$.
  We note that it is easy to construct an Ahlfors function with a logarithmic
  tract which does not satisfy this stronger condition, but to which the
  theorem applies.
  For example, we can consider a function of the form
  $j\circ\phi$, where $j:\D\to\Ch\setminus\{0,1,\infty\}$
  is the elliptic modular function, while $\phi:\D\to\C$ is a
  conformal map which takes a suitable boundary point to a 
  nontrivial prime end whose impression contains the asymptotic value
  $\infty$.

 For a more natural example, 
  suppose that $f:\D\to\C$ is a Valiron function (as discussed
  in Section \ref{sec:defns}) which has an asymptotic value over $\infty$.
  (Such a function can easily be constructed by the same technique as in
  \cite[Example 1]{ripponbarthexceptional}.) Then the logarithmic curve itself spirals in to
  the boundary. Postcomposing $f$ with a suitable M\"obius transformation
  which takes $\infty$ to the unit circle gives rise to an Ahlfors
  islands map with the aforementioned property. 

 The proof of our more general version is analogous
  to that in 
  \cite{baranskikarpinskazdunik};
  the underlying
  ideas can be traced to Stallard's original argument
  from \cite{stallardentirehausdorff}. We will briefly
  explain one step in the proof which differs slightly 
  in our setting.
 
 Consider an annulus $A$ centered at $a$ (in some fixed local
  coordinates), of sufficiently large modulus and sufficiently close to $a$. 
  Then, in the setting of \cite{baranskikarpinskazdunik}, 
  the preimage of this annulus under $f$ is an infinite strip $S$
  tending to
  $a$ in both directions. By basic estimates on the hyperbolic metric,
  $S$ must contain points on the outside of the annulus, and hence
  crosses $A$. This fact allows
  Bara\'nski, Karpi\'nska and Zdunik to construct an iterated function
  system (in logarithmic coordinates) whose limit set has Haudorff dimension
  greater than one. 

 With the more general assumptions in Theorem \ref{thm:logtracts}, 
  it is no longer true that the strip $S$ must tend to $a$. (Recall the
  examples discussed above.) However,
  $S$ still crosses the annulus $A$ at least once, and possibly several times. 
  (This 
  follows from the assumptions on $f$ and
  well-known results of geometric function theory, 
  e.g.\ the fact that radial limit points of a conformal map
  are dense in the boundary of its image.) The remainder of the proof
  goes through exactly as in \cite{baranskikarpinskazdunik}; we 
  refer the reader to this article for further details.

\bibliographystyle{hamsplain}
\bibliography{/Latex/Biblio/biblio}

\end{document}